\documentclass[12pt]{amsart}
\usepackage{amsmath, amssymb, amsthm, amsfonts, marginnote, stmaryrd, tikz-cd}
\usepackage[hidelinks]{hyperref}

\newtheorem*{maintheorem}{Main theorem}
\newtheorem{proposition}{Proposition}

\theoremstyle{remark}
\newtheorem*{remark}{Remark}
\newtheorem*{example}{Example}
\newtheorem*{consistency}{Consistency check}

\newcommand\cO{\mathcal{O}}

\def\bA{\mathbf{A}}

\def\bC{\mathbf{C}}

\def\bG{\mathbf{G}}

\newcommand\frakm{\mathfrak{m}}
\newcommand\frn{\mathfrak{n}}

\newcommand{\beq}{\begin{equation}}
\newcommand{\eeq}{\end{equation}}

\DeclareMathOperator{\pt}{pt}
\DeclareMathOperator{\AJ}{AJ}

\DeclareMathOperator{\Spf}{Spf}

\DeclareMathOperator{\sym}{sym}

\DeclareMathOperator{\Eis}{Eis}
\DeclareMathOperator{\SpecEis}{\check{E}is}
\DeclareMathOperator{\gr}{gr}
\DeclareMathOperator{\Sym}{Sym}

\DeclareMathOperator{\Bun}{Bun}
\DeclareMathOperator{\Hom}{Hom}

\DeclareMathOperator{\Mod}{Mod}

\DeclareMathOperator{\QCoh}{QCoh}

\DeclareMathOperator{\ch}{ch}

\DeclareMathOperator{\Spec}{Spec}

\DeclareMathOperator{\dR}{dR}

\DeclareMathOperator{\SL}{SL}

\DeclareMathOperator{\Betti}{Betti}
\DeclareMathOperator{\Loc}{Loc}

\DeclareMathOperator{\Vect}{Vect}
\DeclareMathOperator{\Shv}{Shv}
\DeclareMathOperator{\FSet}{FSet}

\DeclareMathOperator{\colim}{colim}
\DeclareMathOperator{\Whit}{Whit}

\DeclareMathOperator{\Ran}{Ran}
\DeclareMathOperator{\IndCoh}{IndCoh}

\DeclareMathOperator{\op}{op}
\DeclareMathOperator{\Nilp}{Nilp}

\DeclareMathOperator{\restr}{restr}

\swapnumbers
\numberwithin{equation}{section}

\title{Whittaker coefficients of geometric Eisenstein series}
\author{Jeremy Taylor}

\address{Department of Mathematics\\University of California, Berkeley\\Berkeley, CA  94720-3840}
\email{jeretaylor@berkeley.edu}

\begin{document}
\maketitle

\begin{abstract} Geometric Langlands predicts an isomorphism between Whittaker coefficients of Eisenstein series and functions on the moduli space of $\check{N}$-local systems. We prove this formula by interpreting Whittaker coefficients of Eisenstein series as factorization homology and then invoking Beilinson and Drinfeld's formula for chiral homology of a chiral enveloping algebra.\end{abstract}

\section{Introduction}
\subsection{Notation and conventions}
Let $G$ be a simply connected complex reductive group with Langlands dual group $\check{G}$ defined over $k = \bC$. 
Choose a maximal torus $T$ and a Borel subgroup $B$ with unipotent radical $N$. Let $\rho$ be half the sum of the positive coroots. Let $X$ be a smooth projective complex genus $g$ curve. Choose a square root of the canonical bundle on $X$ and form the anticanonical $T$-bundle $\omega^{-\rho}$. All categories and functors are derived.

Let $\sigma$ be a $\check{T}$-local system on $X$, and let $\Loc_{\check{N}}^{\sigma} = \Loc_{\check{B}} \times_{\Loc_{\check{T}}} \sigma$ be the derived moduli stack of $\check{B}$-local systems on $X$ whose underlying $\check{T}$-local system is identified with $\sigma$, see \eqref{BRdR}. A $\check{T}$-local system is called regular if for every coroot the associated rank 1 local system is nontrivial. If $\sigma$ is regular then $\Loc_{\check{N}}^{\sigma}$ is a classical affine scheme isomorphic to a vector space.

Let $K$ be the Hecke $\sigma$-eigensheaf on $\Bun_T$ whose stalk at $\omega^{-\rho}$ twisted by a negative coweight valued divisor $\underline{\lambda} \cdot \underline{x} = \sum \lambda_i x_i$ is \beq \label{CFTStalk}K_{\omega^{-\rho}(-\underline{\lambda} \cdot \underline{x})} = \left(\bigotimes \sigma^{-\lambda_i}_{x_i}\right)[d_T + d^{\lambda}].\eeq Above, $\sigma^{\lambda}_x$ means the fiber at $x$ of the rank 1 local system obtained from $\sigma$ using $\lambda$. Here $d_T = \dim \Bun_T$ and $d^{\lambda} = \langle 2\check{\rho},  4(g-1)\rho +\lambda \rangle$ is the shift appearing in section 6.4.8 of \cite{G}. 

\begin{remark} Let $K' \in \Shv_{\Nilp}(\Bun_T)$ correspond under class field theory to the skyscraper sheaf $k_{\sigma} \in \QCoh(\Loc_{\check{T}})$. The Hecke eigensheaf condition determines $K'$ up to tensoring by a line. Whittaker normalization says that global sections of $k_{\sigma}$ equals the costalk at the trivial $T$-bundle of $K'[d_T]$. Thus $K$ is only noncanonically isomorphic to a shift of $K'$. On the degree $-\lambda - 2(g - 1)\rho$ connected component $\Bun_T^{\lambda}$, there is a canonical identification $K = \omega^{-\rho} K'[d^{\lambda}]$. We translated $K'$ by $\omega^{-\rho}$, which has the effect of tensoring it by a certain line, see section 4.1 of \cite{GWhit}. \end{remark}

The Whittaker or Poincar\'{e} series sheaf $\Whit = r_!\chi^* D \exp = r_! (-\chi)^* \exp[-2]$ on $\Bun_G$ is the pullback then pushforward of the exponential sheaf along \[\bA^1 \xleftarrow{\chi} \Bun_{N^-}^{\omega^{-\rho}} \xrightarrow{r} \Bun_G,\] see 5.4.1 of \cite{FR}. The function $\chi$ is defined in for example \cite{FGV}. The character sheaf $\exp$ on $\bA^1$ is normalized so that its costalks are in degree zero. Up to a shift, its Verdier dual $D \exp$ is the inverse character sheaf. Both $\exp$ and $D\exp$ corepresent shifted vanishing cycles for conic sheaves on $\bA^1$, so the distinction is not so important. In the Betti setting we do not have the exponential D-module. Because $\chi$ is $\bC^{\times}$-equivariant for the $2\rho$ action on $\Bun_{N^-}^{\omega^{-\rho}}$ and the weight 2 action on $\bA^1$, the sheaf defined in 2.5.2 of \cite{NY} serves as a substitute. The Whittaker sheaf does not have nilpotent singular support.

The automorphic and spectral Eisenstein series functors, $\Eis_! = p_!q^*$ and $\SpecEis = \check{p}^{\IndCoh}_* \check{q}^{\IndCoh *}$, are defined by pullback then pushforward along \[\Bun_T \xleftarrow{q} \Bun_B \xrightarrow{p} \Bun_G \quad \text{and} \quad \Loc_{\check{T}} \xleftarrow{\check{q}} \Loc_{\check{B}} \xrightarrow{\check{p}} \Loc_{\check{G}}.\]  All of the above functors are left adjoints, in particular $\check{p}^{\IndCoh}_*$ is a left adjoint because $\check{p}$ is proper.
In \eqref{LanglandsCommute}, the functor $\Eis_!$ is modified according to section 4.1 of \cite{GWhit} or section 6.4.8 \cite{G}. This matches the translation by $\omega^{-\rho}$ and shift by $d^{\lambda}$ built into our definition of $K$ in \eqref{CFTStalk}.

\subsection{Main theorem statement}
The geometric Langlands conjecture is supposed to be compatible with parabolic induction. Moreover the Whittaker functional is expected to correspond under Langlands to global sections on $\Loc_{\check{G}}$, up to a shift by $d_G = \dim \Bun_G$.
Thus commutativity of conjectural diagram
\beq \label{LanglandsCommute}\begin{tikzcd}[cramped, column sep = tiny]
\Shv_{\Nilp}(\Bun_T) \arrow[d, "\Eis_!((\omega^{-\rho} -){[d^{\lambda}]})"'] & \simeq & \QCoh(\Loc_{\check{T}}) \arrow[d, "\SpecEis(-)"] \\
\Shv_{\Nilp}(\Bun_G) \arrow[dr, "\Hom(\Whit{, }-) {[d_G]}"'] & \simeq & \IndCoh_{\Nilp}(\Loc_{\check{G}}) \arrow[dl, "\Gamma^{\IndCoh}(-)"]\\
& \Vect & \\
\end{tikzcd}\eeq applied to the skyscraper $k_{\sigma}$, predicts the following isomorphism. 

\begin{maintheorem} \label{main} Let $\sigma$ be a $\check{T}$-local system on $X$ and let $K$ be the Hecke eigensheaf on $\Bun_T$ defined in \eqref{CFTStalk}. Whittaker coefficients of Eisenstein series equals functions on moduli space of $\check{N}$-local systems: \[ \Hom(\Whit, \Eis_! K)[d_G] = \cO(\Loc_{\check{N}}^{\sigma}).\]
\end{maintheorem}

The proof uses a combination of \cite{Ras} and \cite{BG} to relate twisted cohomology of the Zastava space to the formal completion of $\Loc_{\check{N}}^{\sigma}$.

Both sides of the main theorem are coweight graded vector spaces. On the automorphic side, let $K^{\lambda}$ be the restriction to the degree $-\lambda - 2(g -1)\rho$ connected component $\Bun_T^{\lambda}$. On the spectral side, the adjoint $\check{T}$-action on $\check{B}$ induces an action on $\Loc_{\check{N}}^{\sigma}$.

Our results apply for all three versions of geometric Langlands: de Rham, restricted, and Betti. On the automorphic side $\Eis_!K$ is a constructible sheaf, equivalently regular holonomic D-module, with nilpotent singular support, see \cite{Gin}. On the spectral side there are three versions of the moduli space of local systems, all having the same complex valued points. For a unipotent group \beq \label{BRdR} \Loc^{\sigma, \dR}_{\check{N}} = \Loc^{\sigma, \restr}_{\check{N}} =  \Loc^{\sigma, \Betti}_{\check{N}} \eeq coincide by proposition 4.3.3 and section 4.8.1 of \cite{AGKRRV}. 

\begin{remark}
If we replace naive Eisenstein series by compactified Eisenstein series of \cite{BGEis}, then geometric Langlands predicts that $\Hom(\Whit, \Eis_{!*}K')$ should equal global sections of a skyscraper sheaf at $\sigma \in \Loc_{\check{G}}$. This is verified by Gaitsgory in appendix B of \cite{BHKT}.
\end{remark}

\begin{consistency}  If $\sigma$ is a regular then theorem 10.2 of \cite{BG} says that $\Eis_!(K^{\lambda})[d_G - d_B^0]$ is perverse. The Whittaker functional $\Hom(\Whit, -)[d_B^0]$ is exact by \cite{NT} or \cite{FR}, so the automorphic side of the main theorem is concentrated in degree 0. This is consistent with $\Loc_{\check{N}}^{\sigma}$ being a classical scheme if $\sigma$ is regular. Here \beq \label{DimB} d_B^{\lambda} = (g-1)\dim B+ \langle 2\check{\rho},  \lambda + 2(g-1)\rho \rangle = \dim \Bun_B^{\lambda}\eeq is the dimension of the degree $-\lambda - 2(g - 1) \rho$ connected component. \end{consistency}

\subsection{Proof outline}
It is convenient to take the coweight graded linear dual to avoid topological rings and because Lie algebra homology behaves better than Lie algebra cohomology.
Here is the proof of our main theorem in one sentence:
\begin{equation} \begin{aligned}\label{Overview}
\Hom&(\Whit, \Eis_! K)^*[-d_G]  \stackrel{\eqref{ZastavaBaseChange}}{=} \bigoplus_{\lambda}\Hom(\chi_Z^* D \exp, q_Z^! DK^{\lambda})[d_T + d^0]  \\ &\stackrel{\eqref{PushConfiguration}}{=}   \bigoplus_{\lambda} \Gamma(X^{\lambda}, \Upsilon^{\lambda}_{\sigma}) \stackrel{\eqref{UpsilonFact}}{=}  \Gamma(\Ran, C_{\bullet}(\check{\frn}_{\sigma}))  \stackrel{\eqref{BDFormula}}{=} C_{\bullet}(\Gamma(X, \check{\frn}_{\sigma})) \stackrel{\eqref{LieDualFunctions}}{=}  \cO(\Loc_{\check{N}}^{\sigma})^*. 
\end{aligned} \end{equation}

In section \ref{BaseZastava}, we use \cite{NT} or \cite{FR} to exchange $\Eis_!$ for a right adjoint, then apply base change and a result of \cite{AG} to get a calculation on the Zastava space. In section \ref{PushConfig}, we pushforward  to the space of positive coweight valued divisors and, by theorem 4.6.1 of \cite{Ras}, obtain a certain factorizable perverse sheaves $\Upsilon_{\sigma}^{\lambda}$ on $X^{\lambda}$.

In section \ref{Enveloping}, we interpret $\Upsilon_{\sigma}$ in terms of the chiral enveloping algebra of $\check{\frn}_{\sigma}$ as in \cite{BG}. In section \ref{ChiralHomology}, we explain, following \cite{BG}, how the cohomology of $\Upsilon_{\sigma}$ equals factorization homology of $A = C_{\bullet}(\check{\frn}_{\sigma})$. Beilinson and Drinfeld's formula says factorization homology of $C_{\bullet}(\check{\frn}_{\sigma})$ is Lie algebra homology of $\Gamma(X, \check{\frn}_{\sigma})$.
In section \ref{Koszul}, we study moduli of $\check{\frn}$-local systems using deformation theory. Since $\Gamma(X, \check{\frn}_{\sigma})$ is the shifted tangent complex of $\Loc_{\check{N}}^{\sigma}$, its Lie algebra homology is related the formal completion of  $\Loc_{\check{N}}^{\sigma}$ at $\sigma$. Using that $\Loc_{\check{N}}^{\sigma} = (\Spec R)/\check{N}$ is the quotient of an affine scheme by a unipotent group and using the contracting $\bG_m$-action, we show that $C_{\bullet}(\Gamma(X, \check{\frn}_{\sigma})) = \cO(\Loc_{\check{N}}^{\sigma})^*$ is the graded linear dual ring of functions.

The idea of using factorization homology to study the formal completion of $\Loc_{\check{N}}^{\sigma}$ is from \cite{BG}. For $\sigma$ regular, propositions 11.3 and 11.4 of \cite{BG} give an isomorphism between $\prod \Gamma(X^{\lambda}, \Upsilon_{\sigma}^{\lambda})^*$ and the completed ring of functions $\cO(\Loc_{\check{N}}^{\sigma})^{\wedge}$. Sections \ref{Enveloping} and \ref{ChiralHomology} review some of their arguments and do not contain new content apart from filling in some details. Our main contribution is in section \ref{Koszul} where we extend the results of \cite{BG} to the more interesting case of irregular $\sigma$, and  obtain a formula for the ring of functions on $\Loc_{\check{N}}^{\sigma}$ (not just its formal completion) using the contracting $\bG_m$-action.

\subsection{Acknowledgements}
I thank David Nadler for suggesting Whittaker coefficients of Eisenstein series and for generous discussions.
This work was partially supported by NSF grant DMS-1646385.

\section{Proof of the main theorem}
\subsection{Base change to Zastava}\label{BaseZastava}
In this section we interpret Whittaker coefficients of Eisenstein series as twisted cohomology of the Zastava space $Z$.

The fiber product $Z' = \Bun_B \times_{\Bun_G} \Bun_{N^-}^{\omega^{-\rho}}$ has a stratification indexed by the Weyl group, determined by the generic relative position of two flags. Let $j: Z \hookrightarrow Z'$ be the open inclusion of the locus where the two flags are generically transverse, called the Zastava space. 

\[\begin{tikzcd}[cramped, column sep = tiny, row sep = small]
& & Z' \arrow[dr] \arrow[dl] & & \\
& \Bun_B \arrow[dr, "p"'] \arrow[dl, "q"'] & & \Bun_{N^-}^{\omega^{-\rho}} \arrow[dr, "\chi"] \arrow[dl, "r"]& \\
\Bun_T & &  \Bun_G & & \bA^1 \\
\end{tikzcd}\]

Consider the compositions \[q_{Z'}: Z' \rightarrow \Bun_B \rightarrow \Bun_T \quad \text{and} \quad  \chi_{Z'}: Z' \rightarrow \Bun_{N^-}^{\omega^{-\rho}} \rightarrow \bA^1\] and let $q_Z = q_{Z'}j$ and $\chi_Z = \chi_{Z'}j$ be their restrictions to $Z$.

\begin{proposition}\label{ZastavaProp}
There is an isomorphism \beq \label{ZastavaBaseChange}\Hom(\Whit, \Eis_! K^{\lambda})^*[-d_G]  = \Hom(\chi_Z^* D\exp, q_Z^! DK^{\lambda})[d_T + d^0].\eeq
\end{proposition}
\begin{proof}
We cannot directly apply adjunction to calculate Whittaker coefficients of Eisenstein series because $\Eis_!$ is a left not right adjoint.
It is shown in \cite{FR} and \cite{NT} that the shifted Whittaker functional $\Hom(\Whit, -)[d_B^0]$ on nilpotent sheaves commutes with Verdier duality $D$. This allows us to exchange $\Eis_! = p_!q^*$ for $\Eis_* = p_*q^!$. Then apply adjunction and base change to reduce to a calculation on the fiber product $Z'$.
\[\Hom(\Whit, \Eis_! K^{\lambda})^*[-2d_B^0]
= \Hom(\Whit, \Eis_*DK^{\lambda}) = \Hom(\chi_{Z'}^* D\exp,  q_{Z'}^!DK^{\lambda})\]
Finally by equation (3.5) of \cite{AG}, restriction to the open generically transverse locus $Z$ does not change the calculation. More precisely the map \[\Hom(\chi_{Z'}^* D\exp, q_{Z'}^!DK^{\lambda}) \xrightarrow{\sim} \Hom(\chi_Z^* D\exp, q_Z^! DK^{\lambda})\] is an isomorphism. For the shifts use \eqref{DimB} and $d_G  + d_T + d^0 = 2d_B^0$.
\end{proof}

\subsection{Pushforward to the configuration space}
\label{PushConfig}
In this section we recall how to factor the projection $q_Z: Z^{\lambda} \rightarrow \Bun_T^{\lambda}$ through the configuration space $X^{\lambda}$ of positive coweight valued divisors of total degree $\lambda$. Hence a description of the $\lambda$-graded piece of proposition \ref{ZastavaProp} as cohomology of a certain perverse sheaf $\Upsilon^{\lambda}_{\sigma}$ on $X^{\lambda}$.

Let $(F, F^-, E) \in Z^{\lambda}$ be a point in the $\lambda$ connected component of Zastava space, that is a $G$-bundle $E$ with generically transverse $B, B^-$-reductions $F, F^-$, such that $F$ has degree $-\lambda - 2(g-1)\rho$ and $F^- \times_{B^-} T = \omega^{-\rho}$. For each dominant weight $\check{\mu}$ the Plucker description gives maps \beq \label{Plucker} F^{\check{\mu}}  \rightarrow E^{\check{\mu}} \rightarrow (F^-)^{\check{\mu}} = \omega^{-\langle \check{\mu}, \rho \rangle}.\eeq Here $F^{\check{\mu}} = F \times_B \bC_{\check{\mu}}$ is a line bundle and $E^{\check{\mu}} = E \times_G V_{\check{\mu}}$ is the vector bundle associated to the simple $G$-module of highest weight $\check{\mu}$.

By the generic transversality condition, the composition \eqref{Plucker} is nonzero map of line bundles, so $\lambda$ is a non-negative coweight. For each point in the Zastava space, there is a unique positive coweight valued divisor $\underline{x} \cdot \underline{\lambda} \in X^{\lambda}$ such that \eqref{Plucker} factors through an isomorphism $F^{\check{\mu}}(\langle \underline{x} \cdot \underline{\lambda}, \check{\mu} \rangle) = \omega^{-\langle \check{\mu}, \rho \rangle}$. Since $G$ is assumed  simply connected, we can write $\lambda = \sum n_i \alpha_i$ as a sum of simple coroots and $X^{\lambda} = \prod X^{(n_i)}$ is a product of symmetric powers of the curve. Therefore $q_Z$  factors through a map $\pi$ to the configuration space followed by the Abel-Jacobi map, \[q_Z : Z^{\lambda} \xrightarrow{\pi} X^{\lambda} \xrightarrow{\AJ} \Bun_T^{\lambda}, \qquad (E, F, F^-) \mapsto \underline{x} \cdot \underline{\lambda} \mapsto F = \omega^{-\rho}(-\underline{x} \cdot \underline{\lambda}).\] 

Let $\lambda$ be a coweight and $n = \langle \check{\rho}, \lambda \rangle$. Let $\check{\frn}_{\sigma} = \sigma \times_{\check{T}} \check{\frn}$, an $\check{\frn}$-local system on $X$. The Chevalley complex on the coweight graded Ran space gives a $\prod S_{n_i}$ equivariant sheaf on $\prod X^{n_i}$. Symmetrizing by pushing forward along $\prod X^{n_i}/S_{n_i} \rightarrow X^{\lambda}$ gives a perverse sheaf $\Upsilon_{\sigma}^{\lambda}$, see section 4 of \cite{Ras} and equation \eqref{UpsilonEnveloping}. The definition of $\Upsilon_{\sigma}^{\lambda}$ involves the Chevalley differential but the associated graded of $\Upsilon_{\sigma}^{\lambda}$ with respect to the Cousin filtration is easier to describe, see section 3.3 of \cite{BG}. 
The stalk of $\Upsilon_{\sigma}^{\lambda}$ at a positive coweight valued divisor $\underline{x} \cdot \underline{\lambda} \in X^{\lambda}$ is \beq \label{TwistedUpsilon}(\Upsilon^{\lambda}_{\sigma})_{\underline{x} \cdot \underline{\lambda}} = \bigotimes C_{\bullet}(\check{\frn}_{\sigma})^{\lambda_i}_{x_i}.\eeq

\begin{proposition} \label{PushConfigProp}
There is an isomorphism \beq\label{PushConfiguration}\Hom(\chi_Z^* D\exp, q_Z^! DK^{\lambda})[d_T + d^0]  = \Gamma(X^{\lambda}, \Upsilon_{\sigma}^{\lambda}).\eeq
\end{proposition}
\begin{proof}
Pushing forward to the configuration space $X^{\lambda}$, the left of \eqref{PushConfiguration} becomes 
\[\Hom(\pi_! \chi_Z^* D \exp, \AJ^!DK^{\lambda})[d_T + d^0] \\ = \Gamma(\Upsilon^{\lambda} \otimes (\AJ^* K^{\lambda})^*)[d_T + d^{\lambda}] = \Gamma(X^{\lambda}, \Upsilon^{\lambda}_{\sigma}).\]
We used that the configuration space $X^{\lambda}$ is smooth so the dualizing sheaf is a rank 1 local system.
And we used theorem 4.6.1 of \cite{Ras}, which says that \[D \pi_! \chi_{Z}^* D \exp = \pi_*\chi_Z^! \exp = \Upsilon^{\lambda}[d^{\lambda} - d^0].\] Here $d^{\lambda} - d^0 = \dim Z^{\lambda}$ and $\Upsilon^{\lambda}$ is a perverse sheaf on $X^{\lambda}$ with stalks $\Upsilon^{\lambda}_{\underline{x} \cdot \underline{\lambda}} = \bigotimes C_{\bullet}(\check{\frn})^{\lambda_i}.$ 

Under class field theory \eqref{CFTStalk}, the stalks of $\AJ^*K^{\lambda}$ are \[(\AJ^* K^{\lambda})_{\underline{\lambda} \cdot \underline{x}}  = \left(\bigotimes \sigma^{\lambda_i}_{x_i}\right)[d_T + d^{\lambda}]\] and its $*$-pullback to $\prod X^{n_i}$ is the $\prod S_{n_i}$ equivariant rank 1 local system $\boxtimes (\sigma^{\alpha_i})^{\boxtimes n_i}$. 
By the projection formula, tensoring with $(\AJ^* K^{\lambda})^*$ has the effect of twisting $\Upsilon^{\lambda}$ by $\sigma$.
\end{proof}

Combining propositions \ref{ZastavaProp} and \ref{PushConfigProp} shows Whittaker coefficients of Eisenstein series is graded dual to global sections of $\Upsilon_{\sigma}$ on the configuration space.

\subsection{The chiral enveloping algebra as a Chevalley complex}
\label{Enveloping}
The local system $\check{\frn}_{\sigma}$ determines a Lie* algebra on the Ran space. Its Lie algebra homology $A := C_{\bullet}(\check{\frn}_{\sigma})$ is a factorization algebra, related to $\Upsilon_{\sigma}$ by partial symmetrization in \eqref{UpsilonEnveloping}.

A sheaf on the Ran space of $X$ is a collection of sheaves $A_{X^I}$ on each power of the curve $X^I$ compatible under $!$-restriction along all partial diagonal maps, see 4.2.1 of \cite{BDChiral} or 2.1 of \cite{FG}. Recall from section 1.2.1 of \cite{FG} that the category of sheaves on the Ran space admits two tensor products with a map $\otimes^* \rightarrow \otimes^{\ch}$ between them.

Pushing forward along the main diagonal $\Delta: X \rightarrow \Ran$, we can regard $\Delta_* \check{\frn}_{\sigma} \in \Shv(\Ran)$ as a Lie algebra for the $*$-tensor product.
Restricting to $X^2$, the Lie* bracket $(\Delta_* \check{\frn}_{\sigma} \otimes^* \Delta_* \check{\frn}_{\sigma})_{X^2} = \check{\frn}_{\sigma} \boxtimes \check{\frn}_{\sigma} \rightarrow (\Delta_* \check{\frn}_{\sigma})_{X^2} = \Delta_* \check{\frn}_{\sigma}$ comes by adjunction from the Lie bracket.

Let $A := C_{\bullet}(\check{\frn}_{\sigma})  \in \Shv(\Ran)$ be Lie algebra homology of $\Delta_* \check{\frn}_{\sigma}$ with respect to the $*$-tensor product, viewed by the forgetful functor as a cocommutative coalgebra with respect to the $\ch$-tensor product. 
Proposition 6.1.2 of \cite{FG} says that $A$ corresponds to the chiral enveloping algebra of $\Delta_* \check{\frn}_{\sigma}$ under the equivalence between factorization and chiral algebras.


The Chevalley complex $A = \bigoplus A^{\lambda}$ is coweight graded because $\Sym (\check{\frn}_{\sigma}[1])$ is coweight graded, and because the Chevalley differential preserves the grading. Choose a coweight $\lambda$ and let $n = \langle \check{\rho}, \lambda \rangle$. The sheaf $A^{\lambda}_{X^n}$ on $X^n$ is $S_n$-equivariant and perverse. Symmetrize it along $\sym: X^n \rightarrow X^{(n)}$ to get a perverse sheaf $(\sym_*A^{\lambda}_{X^n})^{S_n}$ on the $n$th symmetric power. (In other words we pushed forward $A^{\lambda}_{X^n}$ from the stack quotient $X^n/S_n$ to the coarse quotient $X^{(n)}$.)

Now we describe a certain perverse subsheaf $A^{\lambda}_{X^{(n)}} \subset (\sym_*A^{\lambda}_{X^n})^{S_n}$ defined in section 3 of \cite{BG}. 
Let $X_i^{(n)} \subset X^{(n)}$ be the space of effective degree $n$ divisors supported at exactly $i$ points. The $!$-restriction of $(\sym_*A^{\lambda}_{X^n})^{S_n}$ to $X_i^{(n)}$ is a local system whose stalk at a divisor $\underline{n} \cdot \underline{x} \in X_i^{(n)}$ is given by \[(\sym_*A^{\lambda}_{X^n_i})^{S_n}_{\underline{n} \cdot \underline{x}} = \bigoplus_{\lambda = \sum \lambda_j} \bigotimes C_{\bullet}(\frn_{\sigma})^{\lambda_j}_{x_j}.\] The $!$-restriction of $A^{\lambda}_{X^{(n)}}$ to $X^{(n)}_i \subset X^{(n)}$ is the summand whose stalks are \beq \label{SymAStalks} (A^{\lambda}_{X^{(n)}_i})_{\underline{n} \cdot \underline{x}} = \bigoplus_{\substack{\lambda = \sum \lambda_j, \\ \langle \check{\rho}, \lambda_j \rangle = n_j}} \bigotimes C_{\bullet}(\check{\frn}_{\sigma})_{x_j}^{\lambda_j}.\eeq

By section 11.6 of \cite{BG}, the pushforward of $\Upsilon_{\sigma}^{\lambda}$, see \eqref{TwistedUpsilon}, along the partial symmetrization map $\sym^{\lambda}:X^{\lambda} \rightarrow X^{(n)}$ is \beq\label{UpsilonEnveloping}\sym^{\lambda}_*\Upsilon_{\sigma}^{\lambda} =  A^{\lambda}_{X^{(n)}}.\eeq

\subsection{Factorization homology}
\label{ChiralHomology}
In this section we review, following \cite{BG}, how factorization homology of $A^{\lambda} = C_{\bullet}(\check{\frn}_{\sigma})^{\lambda}$ can be computed as cohomology on the symmetric power $X^{(n)}$, where $n = \langle \check{\rho}, \lambda \rangle$.

Let $\FSet$ be the category whose objects are finite nonempty sets and whose morphisms are surjective maps. For each surjection $J \twoheadrightarrow I$ there is a partial diagonal map $\Delta: X^I \rightarrow X^J$. By definition of a sheaf on the Ran space $A_{X^I} = \Delta^!A_{X^J}$ so adjunction gives maps $\Delta_* A_{X^I} \rightarrow A_{X^J}$. Factorization homology is defined in section 6.3.3 of \cite{FG} or section 4.2.2 of \cite{BDChiral} as the colimit over these maps \[\Gamma(\Ran, A) = \underset{\FSet^{\op}}{\colim} \Gamma(A_{X^I}).\]
The following proposition is stated in 11.6 of \cite{BG} and below we fill in the proof using the Cousin filtration and ideas from section 4.2 of \cite{BDChiral}.

\begin{proposition}
The cohomology of $\Upsilon_{\sigma}$, see \eqref{TwistedUpsilon}, is the factorization homology of the Chevalley complex,
\beq \label{UpsilonFact}\bigoplus_{\lambda} \Gamma(X^{\lambda}, \Upsilon_{\sigma}^{\lambda}) = \Gamma(\Ran, A).\eeq
\end{proposition}
\begin{proof}
Equation \eqref{UpsilonEnveloping} relates $\Upsilon_{\sigma}$ to the symmetrization of $A$.
Thus it suffices to show that \beq\label{ConfFact}\Gamma(X^{\lambda}, \Upsilon_{\sigma}^{\lambda}) = \Gamma(A^{\lambda}_{X^{(n)}}) \rightarrow \Gamma(A^{\lambda}_{X^n}) \rightarrow \Gamma(\Ran, A^{\lambda})\eeq is an isomorphism for $n = \langle \check{\rho}, \lambda \rangle$. Indeed we will prove that \eqref{ConfFact} is compatible with the Cousin filtration and that it induces an isomorphism on the associated graded pieces.

Consider the filtration on \eqref{ConfFact} whose $\leq i$th filtered piece consists of sections supported on the partial diagonals of dimensions $\leq i$. The $i$th graded piece is \beq \label{CousinGraded}\Gamma(A^{\lambda}_{X^{(n)}_i}) \rightarrow \Gamma(A^{\lambda}_{X^n_i}) \rightarrow \underset{\FSet^{\op}}{\colim} \Gamma(A^{\lambda}_{X^I_i}) = \gr_i \Gamma(\Ran, A^{\lambda}).\eeq Here $A^{\lambda}_{X_i^{(n)}}$ is the !-restriction of $A^{\lambda}_{X^{(n)}}$ to the space $X_i^{(n)} \subset X^{(n)}$  of effective degree $n$ divisors supported at exactly $i$ points.  Similarly $A^{\lambda}_{X_i^I}$   is the !-restriction of $A^{\lambda}_{X^I}$ to the space $X_i^I \subset X^I$  of $I$-tuples supported at exactly $i$ points. 

The symmetric group $S_i$ acts freely on the space $X^i_i \subset X^i$ of distinct $i$-tuples of points. By section 4.2.3 of \cite{BDChiral}, the $i$th graded piece of the factorization homology of $A^{\lambda}$ is $\gr_i \Gamma(\Ran, A^{\lambda}) = \Gamma(A^{\lambda}_{X^i_i})_{S_i}$.

The connected components of $X^{(n)}_i$ are indexed by partitions $\underline{n} = n_1 + \dots n_i$. Also the local system $A^{\lambda}_{X^i_i}$ splits as a direct sum indexed by such partitions. 
Restricting \eqref{CousinGraded} to the connected component $X^{(\underline{n})}_i \subset X^{(n)}_i$ indexed by a certain partition, \[\Gamma(A^{\lambda}_{X^{(\underline{n})}_i}) \rightarrow \gr_i \Gamma(\Ran, A^{\lambda}) = \Gamma(A^{\lambda}_{X^i_i})_{S_i}\] is an isomorphism onto the corresponding summand of $\Gamma(A^{\lambda}_{X^i_i})_{S_i}$ by \eqref{SymAStalks}.
Summing over partitions shows that the $i$th graded piece of \eqref{ConfFact} is an isomorphism.
\end{proof}

Since the factorization algebra $A = C_{\bullet}(\check{\frn}_{\sigma})$ corresponds to the chiral enveloping algebra $U(\check{\frn}_{\sigma})$, Beilinson and Drinfeld's formula for chiral homology of an enveloping algebra, see theorem 4.8.1.1 of \cite{BDChiral} or 6.4.4 of \cite{FG}, says \beq \label{BDFormula}\Gamma(\Ran, A) = C_{\bullet}(\Gamma(X, \check{\frn}_{\sigma})).\eeq

\subsection{Deformation theory}
\label{Koszul}
In this section we show that \beq \label{LieDualFunctions}C_{\bullet}(\Gamma(X, \check{\frn}_{\sigma})) = \cO(\Loc_{\check{N}}^{\sigma})^*,\eeq Lie algebra homology of the shifted tangent complex equals the graded dual ring of functions on $\Loc_{\check{N}}^{\sigma}$. Deformation theory says that $C_{\bullet}(\Gamma(X, \check{\frn}_{\sigma})) = \Gamma^{\IndCoh}(\omega_{(\Loc_{\check{N}}^{\sigma})^{\wedge}})$ is global sections of the dualizing sheaf on the formal completion at $\sigma$. Using the structure of $\Loc_{\check{N}}^{\sigma}$ described in proposition \ref{AffineUnipotent}, we can recover the graded dual ring of functions on $\Loc_{\check{N}}^{\sigma}$, not just its completion, from $\Gamma^{\IndCoh}(\omega_{(\Loc_{\check{N}}^{\sigma})^{\wedge}})$.

First we show that $\Loc_{\check{N}}^{\sigma} = \Loc_{\check{N}}^{\sigma, x}/\check{N}$ is the quotient by a unipotent group of an affine derived scheme with a contracting $\bG_m$-action. Let $\Loc_{\check{B}}^x = \check{B}^{2g} \times_{\check{B}} 1$ (respectively $\Loc_{\check{T}}^x = \check{T}^{2g} \times_{\check{T}} 1$) be the Betti moduli of $\check{B}$ (respectively $\check{T}$) local systems trivialized at a point $x$.
Let $\Loc_{\check{N}}^{\sigma, x} = \Loc_{\check{B}}^x \times_{\Loc_{\check{T}}^x} \sigma$ be the moduli of $\check{B}$-local systems with underlying $\check{T}$-local system identified with $\sigma$, plus a $\check{T}$-reduction at $x$.  

Since $\check{T}$ is abelian, it acts by automorphisms on $\sigma \in \Loc_{\check{T}}$ so there is a canonical lift $\sigma \in \Loc_{\check{T}}^x$. We also sometimes regard $\sigma$ as a point in $\Loc_{\check{N}}^{\sigma, x}$ via the inclusion $\check{T} \subset \check{B}$.

Let $\check{B}$ act on $\Loc_{\check{B}}^x$ by changing the trivialization at $x$, equivalently by the adjoint action on $\check{B}^{2g} \times_{\check{B}} 1$. 
Restricting the adjoint action along $\check{\rho}$ gives a $\bG_m$-action that contracts $\check{B}$ to $\check{T}$.
Thus we expect a $\bG_m$-action that contracts $\Loc_{\check{B}}^x$ to $\Loc_{\check{T}}^x$, as is made precise below.

\begin{proposition}\label{AffineUnipotent}
The moduli space $\Loc_{\check{N}}^{\sigma, x} = \Spec R$ is a finite type affine scheme with a $\check{B}$-action. Restricting the action along $\check{\rho}$ gives a non-negative grading $R = \bigoplus_{n \geq 0} R_n$ such that $\sigma = \Spec R/R_{> 0}$ is cut out by the ideal of strictly positively graded functions.
\end{proposition}
\begin{proof}
We argue in the Betti setting, but the restricted and de Rham versions also follow by \eqref{BRdR}. First rewrite \beq \label{LocNx}\Loc_{\check{N}}^{\sigma, x} = \Loc_{\check{B}}^x \times_{\Loc_{\check{T}}^x} \sigma = \check{B}^{2g} \times_{\check{T}^{2g} \times_{\check{T}} \check{B}} \sigma = (\check{B}^{2g} \times_{\check{T}^{2g}} \sigma) \times_{\check{B} \times_{\check{T}} 1} 1 = \Spec (R' \otimes_S k).\eeq The contracting $\check{\rho}$-action induces non-negative gradings on the classical rings $R' = \cO(\check{B}^{2g} \times_{\check{T}^{2g}} \sigma)$ and $S = \cO(\check{B} \times_{\check{T}} 1) = \cO(\check{N})$. Since $\check{N}$ is smooth, the augmentation module $k = S/S_{> 0}$ admits a finite graded resolution by free $S$-modules, with all but one term shifted into strictly positive $\check{\rho}$-gradings. Therefore $R = R' \otimes_S k$ is a finite type non-negatively graded ring and $\sigma = \Spec R/R_{>0}$.
\end{proof}

Now we review some derived deformation theory. Let $Y^{\wedge}$ be the formal completion of a derived stack $Y$ at a point $\sigma$. The shifted tangent bundle $T_{\sigma}Y[-1]$ is a dg Lie algebra whose enveloping algebra is endomorphisms of the skyscraper at $\sigma$. By chapter 7 of \cite{GR} or remark 2.4.2 of \cite{Lur}, there is an equivalence \[\Mod(T_{\sigma}Y[-1]) = \IndCoh(Y^{\wedge})\] between Lie algebra modules for the shifted tangent complex and indcoherent sheaves on the formal completion. Let $p: Y^{\wedge} \rightarrow \pt$ be the map to a point. By chapter 7 section 5.2 of \cite{GR}, the trivial $T_{\sigma}Y[-1]$-module corresponds to the dualizing sheaf $\omega_{Y^{\wedge}} = p^! k \in \IndCoh(Y^{\wedge})$. Moreover Lie algebra homology corresponds to global sections
\beq \label{LieHomology} C_{\bullet}(T_{\sigma}Y[-1]) = \Gamma^{\IndCoh}(\omega_{Y^{\wedge}}).\eeq

Suppose $Y^{\wedge} = \Spec R$ is the spectrum of an Artinian local ring $R$. By properness, $p^!$ is right adjoint to $p_*^{\IndCoh}$. 
Therefore the dualizing complex $\omega_{Y^{\wedge}} = R^*$ is the linear dual of $R$ viewed as an $R$-module.

Suppose $Y^{\wedge} = \Spf R^{\wedge} = \colim Y_n$ where $Y_n = \Spec R/\frakm^n$ and let $i_n: Y_n \rightarrow Y^{\wedge}$. Since $Y_n \rightarrow Y_{n+1}$ is proper, $\IndCoh(Y^{\wedge})$ is the colimit under $*$-pushforward of $\IndCoh(Y_n)$, see chapter 1 proposition 2.5.7 of \cite{GR1}.
The dualizing sheaf can be written as a colimit, $\omega_{Y^{\wedge}} = \colim i_{n*}^{\IndCoh} \omega_{Y_n}$, see chapter 7 corollary 5.3.3 of \cite{GR}. Since $\Gamma^{\IndCoh}(Y^{\wedge}, -)$ is continuous it follows that \beq \label{CompleteRing} \Gamma^{\IndCoh}(\omega_{Y^{\wedge}}) = \colim ((R/\frakm^n)^*) = (R^{\wedge})^*\eeq is the \textit{topological} dual of the completed local ring $R^{\wedge}$. In this case, equation \eqref{LieHomology} is corollary 5.2 of \cite{Hin}.

\begin{proposition}\label{DualCompletion}
Let $R = \bigoplus_{n \geq 0} R_n$ be a non-negatively graded finite type derived ring with $R_0 = k$. Let $R^{\wedge}$ be the formal completion with respect to the ideal of positively graded functions. Then the graded dual $R^* = \bigoplus R_n^*$ equals the topological dual of the completion $(R^{\wedge})^*$.
\end{proposition}
\begin{proof}
First suppose $R$ is classical and choose homogeneous generators $f_1, \dots f_r \in R$. Let $d$ be the maximum of their degrees, so $R_{\geq dn} \subset (f_1, \dots f_r)^n \subset R_{\geq n}$. Therefore the graded dual $R^*$ (linear functionals that vanish on some $R_{\geq n}$) equals the topologogical dual $(R^{\wedge})^*$ (linear functionals that vanish on some $(f_1, \dots f_r)^n$).

Now suppose that $R$ is derived. The finite type assumption means that after taking cohomology $H^{\bullet}(R)$ is a finitely generated module over $H^0(R)$, a finitely generated graded classical ring. Choose a finite collection of homogeneous elements $f_1, \dots f_r \in R$ whose images generate $H^0(R)$. 

The formal completion is the topological ring \[R^{\wedge} = R \otimes_{k[f_1, \dots f_r]} k[[f_1, \dots f_r]] = \lim_n R \otimes_{k[f_1, \dots f_r]} (k[f_1, \dots f_r]/k[f_1, \dots f_r]_{> n}).\] For the first equality, see section 6.7 of \cite{GRInd}. The second equality uses that fiber products commute with filtered colimits and that $k[[f_1, \dots f_r]] = \lim (k[f_1, \dots f_r]/k[f_1, \dots f_r]_{> n})$. (The formal completion of a classical positively graded polynomial algebra can be computed using the grading filtration.)

Since $k[f_1, \dots f_r]$ is smooth, $R \otimes_{k[f_1, \dots f_r]}k[f_1, \dots f_r]/k[f_1, \dots f_r]_{>n}$ has finite dimensional cohomology and therefore is concentrated in bounded degrees. Hence for $m$ sufficiently large the quotient map factors through \[R \rightarrow R/R_{> m} \rightarrow R \otimes_{k[f_1, \dots f_r]}k[f_1, \dots f_r]/k[f_1, \dots f_r]_{>n} \rightarrow R/R_{>n}.\] Therefore the formal completion of $R$ can be computed using the grading filtration \[R^{\wedge} = \lim_n R \otimes_{k[f_1, \dots f_r]} (k[f_1, \dots f_r]/k[f_1, \dots f_r]_{> n}) = \lim_n R/R_{> n}.\] 
Taking the topological dual proves $(R^{\wedge})^* = \colim ((R/R_{> n})^*) = \bigoplus R_n^* = R^*$.
\end{proof}

The following proposition shows \eqref{LieDualFunctions}, completing the final step of \eqref{Overview} and the proof of the main theorem.

\begin{proposition}\label{FormalLie}
Lie algebra homology of the shifted tangent complex of $Y =  \Loc_{\check{N}}^{\sigma}$ equals the graded dual of the ring of functions, \[C_{\bullet}(T_{\sigma}Y[-1]) = \cO(Y)^*.\]
\end{proposition}

\begin{proof} 
Write $\Loc_{\check{N}}^{\sigma, x} = \Spec R$ as in proposition \ref{AffineUnipotent}. Let $\check{N}$ act by changing the $\check{T}$-reduction at $x$. Since $\check{T}$ normalizes $\check{N}$, the quotient $Y = \Loc_{\check{N}}^{\sigma} = (\Spec R)/\check{N}$ retains the $\check{\rho}$-action.
The formal completion of $Y$ at $\sigma$ is the inf-scheme $Y^{\wedge} = \Spf(R^{\wedge})/\exp(\check{\frn})$, the quotient by the formal group $\exp(\check{\frn})$. 

Deformation theory says \[C_{\bullet}(T_{\sigma} Y[-1]) = \Gamma^{\IndCoh}(\omega_{Y^{\wedge}}) = ((R^{\wedge})^*)_{\check{\frn}}.\] The first equality is equation \eqref{LieHomology}. For the second equality we pushed forward the dualizing sheaf $\omega_{Y^{\wedge}}$ in two steps, \[Y^{\wedge} \rightarrow \pt/\exp(\check{\frn}) \rightarrow \pt.\] The pushforward of $\omega_{Y^{\wedge}}$ to $\pt/\exp(\check{\frn})$ is an $\check{\frn}$-module. By 
proper base change and \eqref{CompleteRing}, the underlying vector space is $\Gamma^{\IndCoh}(\omega_{\Spf R^{\wedge}}) = (R^{\wedge})^*$ and the $\check{\frn}$-module structure comes from the $\check{N}$-action. Further pushing forward along $\pt/\exp(\check{\frn}) \rightarrow \pt$ corresponds to taking $\check{\frn}$-coinvariants so $\Gamma^{\IndCoh}(\omega_{Y^{\wedge}}) = ((R^{\wedge})^*)_{\check{\frn}}$.
 
Now we show that $\check{\frn}$-coinvariants of the topological dual of $R^{\wedge}$ equals the graded dual ring of functions on $Y$,
\[((R^{\wedge})^*)_{\check{\frn}} = \colim (((R/R_{> n})^*)_{\check{\frn}})  = \colim (((R/R_{> n})^{\check{\frn}})^*)  = \colim (((R^{\check{\frn}} / (R^{\check{\frn}})_{> n})^*) = (R^{\check{N}})^*.\] The ideal $R_{> n}$ is an $\check{\frn}$-module because the $\check{\frn}$-action increases $\check{\rho}$-weights. For the first equality, proposition \ref{DualCompletion} says that $(R^{\wedge})^* = \colim ((R/R_{> n})^*)$, and coinvariants 
commutes with colimits. For the second equality, $((R/R_{> n})^*)_{\check{\frn}} = ((R/R_{> n})^{\check{\frn}})^*$ because $R/R_{> n}$ has finite dimensional cohomology. For the third equality, the image of $(R_{> n})^{\check{\frn}} \rightarrow R^{\check{\frn}}$ is concentrated in degrees $> n$ so we get a map $(R/R_{> n})^{\check{\frn}} \rightarrow R^{\check{\frn}}/(R^{\check{\frn}})_{>n}$. Moreover since $(R/R_{> n})^{\check{\frn}}$ is concentrated in bounded degrees, for $m$ sufficiently large the quotient map factors through \[R^{\check{\frn}} / (R^{\check{\frn}})_{> m} \rightarrow (R/R_{> n})^{\check{\frn}}  \rightarrow R^{\check{\frn}}/(R^{\check{\frn}})_{> n}.\] For the fourth equality, we used the van Est isomorphism, see theorem 5.1 of \cite{H}. Since $\check{N}$ is unipotent, Lie algebra cohomology $R^{\check{\frn}}$ coincides with group cohomology $R^{\check{N}}$.
\end{proof}

\begin{example}
Let $G = \SL(2)$ and let $\sigma$ be a $\check{T}$-local system, viewed as a rank 1 local system using the positive coroot. Then $\sigma$ is regular if and only if it is nontrivial. 

If $\sigma$ is regular, then $\Loc_{\check{N}}^{\sigma} = H^1(X, \sigma)$ is a classical affine scheme because the other cohomologies vanish. The shifted tangent complex $T_{\sigma}\Loc_{\check{N}}^{\sigma}[-1] = H^1(X, \sigma)[-1]$ is an abelian Lie algebra with enveloping algebra $U = \Sym(H^1(X, \sigma)[-1])$. Lie algebra homology of the shifted tangent complex is \[ k \otimes_U k = \Sym H^1(X, \sigma) = \cO(\Loc_{\check{N}}^{\sigma})^*.\]

If $\sigma$ is trivial then $C_{\bullet}(T_{\sigma} \Loc_{\check{N}}[-1]) = \Sym(H^2(X)[-1] \oplus H^1(X) \oplus H^0(X)[1])$ is the graded dual ring of functions on $\Loc_{\check{N}} = H^2(X)[-1] \times H^1(X) \times \pt/H^0(X)$.
\end{example}

\bibliographystyle{alpha}
\bibliography{refs}

\end{document}